\documentclass{amsart}
\usepackage{amsmath,amssymb,amsthm, amscd}
\usepackage[margin=1in]{geometry}
\usepackage{xcolor}
\usepackage{pifont,mathrsfs}
\usepackage{array,lastpage}
\usepackage{enumerate,xspace,pifont,shadow}
\usepackage{mathrsfs}
\usepackage[T1]{fontenc}
\usepackage{mathtools}
\usepackage{algorithm}
\usepackage{algorithmicx}
\usepackage{algpseudocode}
\usepackage{caption}
\usepackage{pgf, tikz, tikz-cd}
\usetikzlibrary[quotes]
\usepackage{hyperref}

\DeclareSymbolFont{AMSb}{U}{msb}{m}{n}
\DeclareMathSymbol{\Z}{\mathbin}{AMSb}{"5A}
\DeclareMathSymbol{\R}{\mathbin}{AMSb}{"52}
\DeclareMathSymbol{\N}{\mathbin}{AMSb}{"4E}
\DeclareMathSymbol{\Q}{\mathbin}{AMSb}{"51}
\DeclareMathSymbol{\F}{\mathbin}{AMSb}{"46}
\DeclareMathSymbol{\C}{\mathbin}{AMSb}{"43}

\newcommand{\avec}{\mathbf{a}}

\newcommand{\fvec}{\mathbf{f}}

\newtheorem{thm}{Theorem}[section]

\newtheorem{cor}[thm]{Corollary}
\newtheorem{prop}[thm]{Proposition}

\newtheorem{fact}[thm]{Fact}

\newtheorem{quest}[thm]{Question}

\theoremstyle{definition}
\newtheorem{definition}[thm]{Definition}

\theoremstyle{remark}
\newtheorem{remark}[thm]{Remark}

\theoremstyle{remark}

\theoremstyle{remark}
\newtheorem{claim}[thm]{Claim}

\theoremstyle{remark}

\theoremstyle{remark}

\theoremstyle{remark}

\begin{document}
\bibliographystyle{amsplain}

\title[Periodic behavior in semigroups]{Periodic behavior in families of numerical and affine semigroups via parametric Presburger arithmetic}
\author{Tristram Bogart, John Goodrick, and Kevin Woods}

\begin{abstract}
Let $f_1(n), \ldots, f_k(n)$ be polynomial functions of $n$. For fixed $n\in\N$, let $S_n\subseteq \N$ be the numerical semigroup generated by  $f_1(n),\ldots,f_k(n)$. As $n$ varies, we show that many invariants of $S_n$ are eventually quasi-polynomial in $n$, such as the Frobenius number, the type, the genus, and the size of the $\Delta$-set. The tool we use is expressibility in the logical system of parametric Presburger arithmetic. Generalizing to higher dimensional families of semigroups, we also examine affine semigroups $S_n\subseteq \N^m$ generated be vectors whose coordinates are polynomial functions of $n$, and we prove similar results; for example, the Betti numbers are eventually quasi-polynomial functions of $n$.
\end{abstract}

\maketitle
\section{Introduction}

In this note we will describe a general context for studying the periodic behavior of formulas for certain invariants of \emph{affine semigroups}, that is, subsets of $\N^m$ that are closed under addition and contain $\textbf{0}$.
In the case when $m =1$ and $\N\setminus S$ is finite, $S$ is known as a \emph{numerical semigroup}.

Numerical (and affine) semigroups can be specified by a set of generators: given $a_1,\ldots,a_k\in\N$, let $\langle a_1,\ldots,a_k\rangle=\{\lambda_1a_1+\cdots+\lambda_ka_k:\ \lambda_i\in\Z_{\ge 0}\}$.
Several recent papers have studied the behavior of \emph{shifted families} of numerical semigroups, that is, families of semigroups $S_n = \langle a_1 + n, a_2 + n, \ldots, a_k +n \rangle$ where $a_1, \ldots, a_k \in \N$ are held constant and $n$ varies over $\N$ (e.g. \cite{CKLNZ} and \cite{CGHOP}). One would like to find nice formulas in terms of $n$ for important quantities associated with $S_n$, such as the Frobenius number, the genus, the type, and the Betti numbers (see below for definitions), among others. Kerstetter and O'Neill \cite{KO} proposed the much more general problem of studying \emph{parametric families} of numerical semigroups $S_n = \langle f_1(n), \ldots, f_k(n)\rangle$ where the $f_i(n):\N\rightarrow\N$ are arbitrary polynomials; we will call these \emph{parametric numerical semigroups}. In the case where the polynomials $f_i(n)$ are all linear, they showed that the first Betti number is eventually periodic. In the case of parametric numerical semigroups of arbitrary degree, they conjecture that the Betti numbers satisfy the following condition:

\begin{definition}
A function $f: \N \rightarrow \Z$ is \emph{eventually quasi-polynomial of degree $d$} (or ``EQP'') if there are eventually periodic functions $c_0, c_1, \ldots, c_d : \N \rightarrow \Q$ such that for all $n \in \N$, $$f(n) = c_d(n) n^d + c_{d-1}(n) n^{d-1} + \ldots + c_0(n)$$ and for infinitely many values of $n$, $c_d(n) \neq 0$.
\end{definition}

We prove that conjecture. In fact, using a general result from \cite{BGW2017} about families definable in parametric Presburger arithmetic (as explained below), we can show that many invariants of parametric numerical semigroups are EQP. Each invariant is defined in Section 3 or 4; see also the recent books on numerical semigroup theory by Assi and Garc\'ia S\'anchez (\cite{Assi_GS}) and Rosales and Garc\'ia S\'anchez (\cite{Rosales_GS}).

\begin{thm}
\label{main_numerical}
Let $\{S_n \, : \, n \in \N\}$ be a parametric family of sub-semigroups of $\N$ such that $S_n = \langle f_1(n), \ldots, f_k(n) \rangle$, where the $f_i : \N \rightarrow \N$ are polynomial functions with integer coefficients.

\begin{enumerate}
\item The following sets are all eventually periodic:
\begin{enumerate}
\item The set of all $n \in \N$ for which $S_n$ is a numerical semigroup (that is, such that $|\N \setminus S | < \infty$).
\item The set of all $n \in \N$ for which $S_n$ is a symmetric numerical semigroup.
\item The set of all $n \in \N$ for which $S_n$ is an irreducible numerical semigroup.
\end{enumerate}

\item The following quantities, whenever finite, are eventually quasi-polynomial:
\begin{enumerate}
\item The \emph{genus} $\textup{g}(S_n)$, defined as $|\N \setminus S_n|$.
\item The \emph{Frobenius number} $\textup{F}(S_n)$, that is, the maximal element of $\N \setminus S_n$.
\item The \emph{type} of $S_n$.
\item The cardinality $|\textup{FG}(S_n)|$ of the set of fundamental gaps of $S_n$. 
\item The cardinality of the \emph{delta set} $\Delta(S_n)$.
\item For any fixed $i$, the value of the $i$-th element of the Ap\'ery set of $S_n$ as listed in increasing order.
\end{enumerate}

\item Let $a: \N \rightarrow \N$ be given by a polynomial function. Then the following quantities are all EQP:

\begin{enumerate}
\item The cardinality of the length set $\textup{L}_{S_n}(a(n))$ containing all possible factorizations of $a(n)$ in $S_n$. (In case $a(n) \notin S_n$ we consider $\textup{L}_{S_n}(a(n))$ to be empty.)
\item The cardinality of the delta set $\Delta_{S_n}(a(n))$ consisting of all differences between successive elements in the length set  $\textup{L}_{S_n}(a(n))$ (which we take to be the empty set in case  $|\textup{L}_{S_n}(a(n))| \leq 1$).
\end{enumerate}

\end{enumerate}

\end{thm}

The EQP behavior of the Frobenius number $\textup{F}(S_n)$ and of the genus $\textup{g}(S_n)$ were established previously by Bobby Shen (see \cite{Shen15a}), but we include it on this list above for completeness. As far as we know, most of the other EQP and eventually periodic phenomena in the theorem above (such as the Betti numbers) have not been noted previously. 

\begin{remark}
This list of EQP behavior is far from exhaustive. The parametric Presburger arithmetic tools that we describe in the next section are quite general and could be used to prove other EQP behavior in numerical semigroups (the cardinality of the set of special gaps, the minimum pseudo-Frobenius number, and so on).
\end{remark}

\begin{remark}
While Theorem \ref{main_numerical} lists EQP behavior for several \emph{numerical} invariants of $S_n$, there are also a number of important \emph{sets} associated with $S_n$, such as $S_n$ itself, $\N\setminus S_n$, the set of  pseudo-Frobenius numbers, the fundamental gaps, the delta set, the Ap\'ery set, and the length sets. It is harder to discuss the ``periodic'' behavior of the sets themselves (rather than only concentrating on the cardinality, the smallest element, or such). However, by encoding sets $T_n$ as generating functions $g_n(x)=\sum_{t\in T_n}x^t$, the tools we will use here similarly yield a form of periodic behavior for these generating functions $g_n$, as $n$ varies; we don't discuss generating functions further herein, but see \cite{BGW2017} for more details.
\end{remark}

One may also consider the even more general context of \emph{parametric affine semigroups}: given polynomial functions $\textbf{f}_1(n), \ldots, \textbf{f}_k(n) : \N \rightarrow \N^m$, let $S_n$ be the sub-semigroup of $\N^m$ generated by $\textbf{f}_1(n), \ldots, \textbf{f}_k(n)$. Let $K$ be a field. For each value of $n$ we can consider a minimal graded free resolution of the semigroup algebra $K[S_n]$. By definition $K[S_n]$ is a monomial algebra and thus the kernel $I_{S_n}$ of the first map in the resolution is a binomial prime ideal, or \emph{toric ideal}. Such ideals have applications in integer programming and geometric combinatorics \cite{Sturmfels} as well as algebraic statistics \cite[\S 9]{Sullivant}. The $i$th Betti number $\beta_i$ of $K[S_n]$ counts the number of $i$th syzygies of $I_{S_n}$, that is, the size of any minimal generating set of the $i$th kernel map in the minimal free resolution. In particular, the first Betti number $\beta_1$ is the size of a minimal generating set for the ideal $I_{S_n}$.

\begin{thm}
\label{main_affine}
Let $S_n = \langle \textbf{f}_1(n), \ldots, \textbf{f}_k(n)\rangle \subseteq \N^m$ be a parametric affine semigroup. Then for each $i$ and every field $K$, the $i$-th Betti number of the semigroup algebra $K[S_n]$ is eventually quasi-polynomial.
\end{thm}
As we will briefly explain at the beginning of Section 4, the following is an immediate consequence. 
\begin{cor} \label{presentation}
The cardinality of a minimal presentation of a parametric affine semigroup $S_n$ is eventually quasi-polynomial. 
\end{cor}

All of the results above will follow from the main theorem of our previous paper \cite{BGW2017} relating eventually periodic and EQP behavior to expressibility in a logical framework known as \emph{parametric Presburger arithmetic}, which we will lay out in Section 2. In Section 3 we will at once define the various invariants that appear in Theorem~\ref{main_numerical} and prove the theorem. In Section 4 we will combine our technique with a homological interpretation (see \cite{BrunsHerzog}) of the Betti numbers of semigroup algebras in order to prove Theorem~\ref{main_affine}. In Section 5, we end with some discussion and questions about the degrees of the EQP functions that may arise as the cardinalities of delta sets and as Betti numbers.

\textbf{Acknowledgments:} The authors would like to thank Mauricio Velasco for useful and encouraging discussions of some of the ideas presented here while we were preparing this paper. The first and second authors were respectively supported by internal research grants INV-2017-51-1453 and INV-2018-50-1424 from the Faculty of Sciences of the Universidad de los Andes during their work on this project.  

\section{Parametric Presburger arithmetic: logic and counting formulas}
In this section we describe the main tool we use to prove the theorems above: parametric Presburger arithmetic. To make this note as self-contained as possible, we repeat the definition of parametric Presburger formulas introduced in \cite{Woods2014} in some detail and restate the main result of \cite{BGW2017} (Theorem~\ref{BGW} below). The reader with a background in logic may skip ahead to Definition~\ref{ParamPresForm}.

The system of parametric Presburger arithmetic is a generalization, introduced by Woods (\cite{Woods2014}), of the classical theory of Presburger arithmetic, that is, the first-order logical theory of the structure $(\Z, +, <)$ . First, we recall the definition of ``classical'' (non-parametric) Presburger formulas. We will follow standard notational conventions in mathematical logic; see, for example, \cite{Marker} or any other recent textbook on logic for a more detailed discussion of the notion of a first-order formula.

\begin{definition}
\label{PresForm}
A \emph{(classical) Presburger formula} is a first-order logical formula that can be written using quantifiers ($\exists$, $\forall$), boolean operations (and, or, not), and integer linear inequalities in the variables. More formally, it is any expression $\varphi$ which can be built up in finitely many steps via the following rules:

\begin{enumerate}
\item For any integers $a_1, \ldots, a_m$ and $b$ and any variables $x_1, \ldots, x_m$, the expression $$\sum_{i=1}^m a_i x_i \leq b$$ is a Presburger formula (formulas as above are called \emph{atomic}). The $x_i$ here are selected from some countably infinite set $\{z_0, z_1, \ldots\}$ of variable symbols.

\item If $\varphi$ and $\psi$ are Presburger formulas, then the expressions $(\varphi \wedge \psi)$, $(\varphi \vee \psi)$, $(\varphi \rightarrow \psi)$, and $(\neg \varphi)$ constructed with the logical connectives $\wedge$ (``and''), $\vee$ (``or''), $\rightarrow$ (``implies'') and $\neg$ (``not'') are also Presburger formulas. (The parentheses are included here only to avoid ambiguity.)
\item If $\varphi$ is a Presburger formula and $z_i$ is any variable, then $\exists z_i \varphi$ and $\forall z_i \varphi$ are also Presburger formulas.
\end{enumerate}

We use the notation $\varphi(x_1, \ldots, x_m)$ to denote a Presburger formula whose \emph{free variables} (those which do not occur within the scope of any quantifier $\forall$ or $\exists$) are included in the set $\{x_1, \ldots, x_m\}$.

\end{definition}

\begin{definition}
\label{PresDef}
We call a set $X \subseteq \Z^m$ \emph{(classically) Presburger definable} if $S$ is the set of all tuples in $\Z^m$ which satisfy some Presburger formula $\varphi(x_1, \ldots, x_m)$ under the natural interpretation of the logical symbols and the non-logical symbol $\leq$ and considering the quantifiers $\exists z_i$ and $\forall z_i$ to range over values in $\Z$. In this case, $\varphi(x_1, \ldots, x_m)$ is said to \emph{define} the set $X$, and we write $X = \varphi(\Z^m)$.
\end{definition}

Our definition of ``Presburger formula'' differs slightly in its syntax from what the reader may find in a logic textbook, but all definitions will agree on which subsets of $\Z^m$ count as ``Presburger definable,'' and this definability is really the key concept for us.

For example, if $a_1, \ldots, a_m$ and $b$ are any integers, then the set of all $(x_1, \ldots, x_m) \in \Z^m$ satisfying the inequality $$\sum_{i=1}^m a_i x_i \leq b$$ is Presburger definable. The set of integer solutions to the equation $$\sum_{i=1}^m a_i x_i = b$$ is also Presburger definable since it is the set of all $(x_1, \ldots, x_m)$ satisfying the Presburger formula $$\left(\sum_{i=1}^m a_i x_i \leq b \right) \wedge \neg \left( \sum_{i=1}^m a_i x_i \leq b - 1\right),$$ and similarly subsets of $\Z$ defined as solutions to finite lists of equations and inequalities with coefficients in $\Z$ are Presburger definable.

An important example for us is that any numerical semigroup $S = \langle a_1, \ldots, a_k \rangle$ is Presburger definable: it is the set of all $x \in \Z$ satisfying the formula $$x \geq 0 \wedge \exists z_1 \ldots \exists z_k \, \left( z_1 \geq 0 \wedge \ldots \wedge z_k \geq 0 \wedge \sum_{i=1}^k a_i  z_i = x  \right).$$

Now we come to the system of parametric Presburger arithmetic defined by Woods in \cite{Woods2014}. This is a logical system in which we can define \emph{parameterized families} $\{ S_n \subseteq \Z^m \, : \, n \in \N\}$ of subsets of $\Z^m$; we reserve the letter $n$ for a special parameter variable, and we seek to understand how the set $S_n$ changes as $n$ grows. 

The definition of a parametric Presburger formula $\varphi_n(x_1, \ldots, x_m)$ is just like that of a classical Presburger formula except that we allow the terms $a_i$ and $b$ in the atomic formulas (clause (1)) to be polynomial functions of the parameter $n$. More precisely:

\begin{definition}
\label{ParamPresForm}
A \emph{parametric Presburger formula} is any expression $\varphi_n$ which can be built up in finitely many steps via the following rules:

\begin{enumerate}
\item For any polynomials $a_1(n), \ldots, a_m(n)$ and $b(n)$ from $\Z[n]$ and any variables $x_1, \ldots, x_m$, the expression $$\sum_{i=1}^m a_i(n) x_i \leq b(n)$$ is a parametric Presburger formula (formulas as above are called \emph{atomic}). The $x_i$ here are selected from some countably infinite set $\{z_0, z_1, \ldots\}$ of variable symbols which does \textbf{not} include $n$, a symbol reserved for the parameter.

\item If $\varphi_n$ and $\psi_n$ are parametric Presburger formulas, then the expressions $(\varphi_n \wedge \psi_n)$, $(\varphi_n \vee \psi_n)$, $(\varphi_n \rightarrow \psi_n)$, and $(\neg \varphi_n)$ constructed with the logical connectives $\wedge$ (``and''), $\vee$ (``or''), $\rightarrow$ (``implies'') and $\neg$ (``not'') are also parametric Presburger formulas.

\item If $\varphi_n$ is a parametric Presburger formula and $z_i$ is any variable symbol \textbf{except} $n$, then $\exists z_i \varphi_n$ and $\forall z_i \varphi_n$ are also parametric Presburger formulas.
\end{enumerate}

We use the notation $\varphi_n(x_1, \ldots, x_m)$ to denote a parametric Presburger formula whose \emph{free variables} (those which do not occur within the scope of any quantifier $\forall$ or $\exists$) are included in the set $\{x_1, \ldots, x_m\}$, and with the subscript $n$ to emphasize the value of the parameter.

\end{definition}

\begin{definition}
\label{ParamPresDef}
A family $\{S_n \, : \, n \in \N\}$ of subsets of $\Z^m$ is \emph{parametric Presburger definable} just in case there is some parametric Presburger formula $\varphi_n(x_1, \ldots, x_m)$ such that $S_n = \varphi_n(\Z^m)$ (that is, $S_n$ is the set defined by $\varphi_n$). The quantifiers $\forall$ and $\exists$ are interpreted to range over $\Z$ as usual.

\end{definition}

We emphasize that each member $S_n$ of a parametric Presburger family is definable in classical Presburger arithmetic. To the reader familiar with classical Presburger arithmetic, it may help to think of parametric Presburger formulas $\varphi_n(x_1, \ldots, x_m)$ as first-order formulas in an extended language with a new special variable ``$n$'' such that we are allowed to multiply terms by $n$, but we are not allowed to quantify over $n$, nor can we multiply together two terms in the standard variables $x_i$.

For example, if $S_n = \langle \textbf{f}_1(n), \ldots, \textbf{f}_k(n)\rangle$ is the parametric affine semigroup generated by functions $\textbf{f}_i : \N \rightarrow \N^m$ where $$\textbf{f}_i(n) = (f_{i,1}(n), \ldots, f_{i,m}(n))$$ for polynomial functions $f_{i,j}: \N \rightarrow \N$ with coefficients in $\Z$, then the family $\{S_n\}$ is parametric Presburger definable, since it is defined by the parametric Presburger formula $$\textbf{x} \geq \textbf{0} \wedge \exists z_1 \ldots \exists z_k \, \left( z_1 \geq 0 \wedge \ldots \wedge z_k \geq 0 \wedge \sum_{i=1}^k z_i \cdot \textbf{f}_i(n) = \textbf{x}  \right) $$ where we use the notation $\textbf{x}$ for $(x_1, \ldots, x_m)$ and ``$\textbf{x} \geq \textbf{0}$'' as shorthand for the conjunction $x_1 \geq 0 \wedge \ldots \wedge x_m \geq 0$.


Now we recall the main result of \cite{BGW2017}:

\begin{thm}
\label{BGW}
(\cite{BGW2017}, Theorem 1.15) Suppose that $\{S_n \, : \, n \in \N \}$ is a parametric Presburger definable family of subsets of $\Z^d$.
\begin{enumerate}
\item The set of $n$ such that $S_n \neq \emptyset$ is eventually periodic.
\item The set of $n$ such that $S_n$ is infinite is eventually periodic.
\item There is an EQP function $g: \N \rightarrow \N$ such that for every $n$ at which $S_n$ is finite, $|S_n| = g(t)$.
\item There is a function $\textbf{a} : \N \rightarrow \Z^m$ whose coordinates are EQP and such that at every $n$ for which $S_n \neq \emptyset$, we have $\textbf{a}(n) \in S_n$.
\end{enumerate}
\end{thm}

\section{Defining factorization invariants in parametric Presburger arithmetic}

In this section, we recall the definition of various frequently studied concepts in semigroup theory, especially factorization invariants for numerical and affine semigroups, and along the way we show how they may be defined in parametric Presburger arithmetic. Combined with Theorem~\ref{BGW}, this will establish Theorem~\ref{main_numerical}.

Let $S_n = \langle f_1(n), \ldots, f_k(n) \rangle$ be a sub-semigroup of $\N$ generated by polynomials $f_i(n)$ with integer coefficients such that $f_i$ maps $\N$ into $\N$. In general, this will not yield a numerical semigroup if the greatest common divisor $d(S_n)$ of $f_1(n), \ldots, f_k(n)$ is greater than $1$. 

Observe that $x = d(S_n)$ may be expressed by the parametric Presburger formula $$ x > 0 \wedge \exists z_1, \ldots, z_k \left[ x = \sum_{i=1}^k f_i(n) \cdot z_i \wedge \forall y \left( \exists z'_1, \ldots, z'_k \left[ y = \sum_{i=1}^k f_i(n) \cdot z'_i \wedge y > 0 \right] \rightarrow x \leq y \right) \right]$$ expressing that $x$ is the least positive $\Z$-linear combination of the generating set. Thus the function sending $n$ to $d(S_n)$ is EQP, and the set of $n$ for which $S_n$ is a numerical semigroup is eventually periodic, establishing (1) (a) of Theorem~\ref{main_numerical}.

For the rest of this section, we will consider only the (eventually periodic set of) values of $n$ for which $S_n$ is a numerical semigroup. Perhaps the simplest invariant of $S_n$ to define in parametric Presburger arithmetic is the \emph{genus} $\textup{g}(S_n) := |\N \setminus S_n|$. Using the fact mentioned in the previous section that $S_n$ itself is parametric Presburger definable, it follows that its complement in $\N$ is parametric Presburger definable as well, giving us (2) (a) of Theorem~\ref{main_numerical}.

Recall that the \emph{Frobenius number} $\textup{F}(S_n)$ is the maximum of $\N \setminus S_n$, so $x = \textup{F}(S_n)$ may be expressed by the parametric Presburger formula
$$x \notin S_n \wedge \forall y \left( x < y \rightarrow y \in S_n \right)$$
(using the fact mentioned in the previous section that $S_n$ itself is parametric Presburger definable), and so $\textup{F}$ is EQP, establishing (2) (b) of Theorem~\ref{main_numerical}. This was previously proved by Bobby Shen \cite{Shen15a} using a different method.


The set $PF(S_n)$ of \emph{pseudo-Frobenius numbers} of $S_n$ is the set of all $x \in \Z$ satisfying the Presburger formula $$x \notin S_n \wedge \forall y \in \Z \left[ (y \in S_n \wedge y \neq 0) \rightarrow x+y \in S_n \right],$$ and the cardinality of $PF(S_n)$ (which is always finite) is known as the \emph{type} of $S_n$. Thus the type of $S_n$ is given by an EQP function, yielding (2) (c) of Theorem~\ref{main_numerical}.

The numerical semigroup $S$ is called \emph{symmetric} if $$\forall z \in \Z \left[ z \notin S \rightarrow \textup{F}(S) - z \in S \right].$$ Since the Frobenius number $\textup{F}(S_n)$ is parametric Presburger definable, it is immediate from this definition that there is a parametric Presburger formula expressing that ``$S_n$ is a symmetric numerical semigroup'' and so the set of $n$ for which this holds is eventually periodic (part (1) (b) of Theorem~\ref{main_numerical}).

The numerical semigroup $S$ is \emph{irreducible} if it is not the intersection of two strictly larger numerical semigroups.

\begin{fact}
\cite{problemlist}
A numerical semigroup $S$ is irreducible if and only if one of the two conditions below holds:
\begin{enumerate}
\item $\textup{F}(S)$ is odd and $S$ is symmetric; or
\item $\textup{F}(S)$ is even and $S$ is \emph{pseudo-symmetric:} that is, for any $x \in \Z \setminus S$, either $\textup{F}(S) - x \in S$ or $x = \frac{\textup{F}(S)}{2}$.
\end{enumerate}
\end{fact}

With this characterization of irreducibility, we now leave it as an exercise to the reader to verify that ``$S_n$ is irreducible'' may be expressed by a parametric Presburger formula.  (Hint: the statement ``$n$ is even'' may be expressed by the parametric Presburger formula $\exists z ( 2 z = n )$, which allows one to break into cases according to the parity of $n$.) Hence the set of all $n$ for which $S_n$ is irreducible is eventually periodic (part (1) (c) of Theorem~\ref{main_numerical}).

A \emph{fundamental gap} of a numerical semigroup $S$ is an integer $x$ such that $x \notin S$ but $k x \in S$ for every integer $k \geq 2$ (see \cite{Assi_GS}).
This is equivalent to $2x$ and $3x$ belonging to $S$, so the parametric Presburger formula $$x \notin S_n \wedge 2 x \in S_n \wedge 3x \in S_n$$ defines the set $\textup{FG}(S_n)$ of fundamental gaps of $S_n$. Thus its cardinality (which is always finite if $S_n$ is a numerical semigroup) is an EQP function of $n$ (part (2) (d) of Theorem~\ref{main_numerical}).

Next we define length sets and delta sets. The \emph{length set} $\textup{L}_{S_n}(m)$ of $m \in \N$ may be defined by the condition that $x \in \textup{L}_{S_n}(m)$ if and only if $$\exists z_1, \ldots, z_k \left[m = \sum_{i=1}^k f_i(n) \cdot z_i \wedge z_1 \geq 0 \wedge \ldots \wedge z_k \geq 0 \wedge x = \sum_{i=1}^k z_i \right],$$ that is, \emph{$\textup{L}_{S_n}(m)$ is the set of all sums of coefficients $z_i$ occurring in factorizations of $m$}. If we list the elements of the length set $\textup{L}_{S_n}(m) = \{\ell_1, \ldots, \ell_r\}$ in increasing order, then the delta set $\Delta_{S_n}(m)$ is the set of all successive differences $\ell_{i+1} - \ell_i$, and the condition that $x \in \Delta_{S_n}(m)$ can be defined by the parametric Presburger formula
$$\exists y_1, y_2 \left[y_1 \in \textup{L}_{S_n}(m) \wedge y_2 \in \textup{L}_{S_n}(m) \wedge y_1 < y_2 \wedge x = y_2 - y_1  \wedge \forall z\left( (z \in \textup{L}_{S_n}(m) \wedge z > y_1 ) \rightarrow y_2 \leq z \right) \right] $$
expressing that $x$ is a difference $y_2 - y_1$ of elements $y_i \in \textup{L}_{S_n}(m) $ such that there is no third element of the length set between them. Finally, the delta set of $S_n$, $$\Delta (S_n) := \bigcup_{m \in S_n} \Delta_{S_n} (m),$$ is also parametric Presburger definable, since $x \in \Delta(S_n)$ is equivalent to
$$\exists y \left[ y \in S_n \wedge x \in \Delta_{S_n}(y) \right].$$
From the preceding discussion we immediately conclude that if $a(n) \in \N$ is given by a polynomial function, then the cardinalities of the length and delta sets ($\textup{L}_{S_n}(a(n)$ and $\Delta_{S_n}(a(n))$ respectively) are EQP (Parts (3) (a) and (3) (b) of Theorem~\ref{main_numerical}), and that the cardinality of the delta set $\Delta(S_n)$ is EQP (Theorem~\ref{main_numerical} part (2) (e)). 

The \emph{Ap\'ery set} of $x \in S_n$ is defined as $$\textup{Ap}(S_n, x) = \{ y \in S_n \, : \, y - x \notin S_n \},$$ and the Ap\'ery set of $S_n$ itself is $\textup{Ap}(S_n) = \textup{Ap}(S_n, m(n))$ where $m(n)$ is the minimal positive element of $S_n$. Note that $|\textup{Ap}(S_n, x)| = x$ since elements of the Ap\'ery set correspond to minimal positive elements of each congruence class modulo $x$ which lie in $S_n$. Now it is simple to define when $x = m(n)$ via the parametric Presburger formula $$x \in S_n \wedge x > 0 \wedge \forall y [ (y \in S_n \wedge y > 0) \rightarrow x \leq y ],$$ and from this we can characterize when $x \in \textup{Ap}(S_n)$ with the parametric Presburger formula $$x \in S_n \wedge x - m(S_n) \notin S_n.$$ If we fix some $i \geq 1$, then $x$ is the $i$-th element of the Ap\'ery set of $S_n$ in increasing order if and only if it satisfies the parametric Presburger formula expressing ``$x$ is in $\textup{Ap}(S_n)$, and there are $i-1$ pairwise distinct elements of $\textup{Ap}(S_n)$ below $x$, and it is not the case that there are $i$ pairwise distinct elements of $\textup{Ap}(S_n)$ below $x$.'' From this it follows that the formula for the $i$-th element of $\textup{Ap}(S_n)$ is EQP (Part (2) (f) of Theorem~\ref{main_numerical}).

\section{Betti numbers of affine semigroup algebras}
Any affine semigroup $\langle \avec_1, \ldots, \avec_k \rangle$ (thus in particular, any numerical semigroup) is isomorphic to a quotient of a free commutative semigroup in the following sense: let $F_k := \langle e_1, \ldots, e_k \rangle$ be the free commutative semigroup with $k$ generators, and define a \emph{congruence relation} to be an equivalence relation $\sim$ on $F_k$ with the property that whenever $x \sim y$ then $x+z \sim y+z$. For such a congruence relation, we can naturally form the quotient semigroup $F_k / \sim$. Any numerical semigroup with $k$ generators is isomorphic to such a quotient $F_k / \sim$ in which moreover $\sim$ is generated by a finite number of relations $a_1 \sim b_1, \ldots, a_r \sim b_r$ (see \cite[Corollary 1.6]{Herzog}). That is, $\sim$ is the intersection of all congruence relations on $F_k$ which contain every relation $a_i \sim b_i$. These relations which generate $\sim$ constitute a \emph{presentation of $S$}, and if $r$ is the minimum possible over all possible such finite generating sets then this is a \emph{minimal presentation of $S$}.

Corollary \ref{presentation} states that the cardinality of the size of a minimal presentation for $S_n$ is an EQP function of $n$. In fact this number is equal to the first Betti number $\beta_1(K[S_n])$ of the semigroup algebra $K[S_n]$ over any field $K$, and the EQP behavior of this function follows from the more general result (Theorem \ref{main_affine}) that any Betti number $\beta_i(K[S_n])$ is EQP, which we will now prove.  The Betti numbers are more complicated than the invariants discussed in Section 3 and do not correspond to parametric Presburger definable sets in an obvious way, but a characterization of Betti numbers by Bruns and Herzog will help us.

For the rest of this section, we work in the context of \emph{parametric affine semigroups}: that is, we assume that $S_n$ is a sub-semigroup of $\N^m$ for some $m$ generated by $\textbf{f}_1(n), \ldots, \textbf{f}_k(n)$, where each $\textbf{f}_i$ is a function from $\N$ to $\N^m$ given coordinate-wise by polynomials with integer coefficients.

We briefly recall the usual way of defining Betti numbers. 
Let $K$ be any field, and let $R = K[x_1, \ldots, x_k]$ be the polynomial ring in $k$ indeterminates. There is a natural $K$-algebra map $\varphi: R \rightarrow K[y_1, \ldots, y_m]$ which sends each $x_i$ to $\textbf{y}^{\textbf{f}_i(n)}$, that is, the monomial in variables $y_1, \ldots, y_m$ such that the exponent of $y_j$ is the $j$-th coordinate of $\textbf{f}_i(n)$. The ring $K[y_1, \ldots, y_m]$ is given its standard multigrading (so degrees are elements of $\Z^m$), and the ring $R$ is given the nonstandard multigrading under which $\deg(x_i) = \deg(\textbf{y}^{\textbf{f}_i(n)})$ so that $\varphi$ is a homogeneous degree-zero map. The image of $\varphi$ is the \emph{affine semigroup algebra} $K[S_n]$ and its kernel is the \emph{toric ideal} $I_{S_n}$. 
 
 Under this multigrading of $R$, we now take a minimal free multigraded resolution $\mathbf{F}$ of $R$, that is, a chain complex $$\mathbf{F} \, : \, 0 \rightarrow F_{k-1} \rightarrow F_{k-2} \rightarrow \ldots \rightarrow F_1 \rightarrow F_0 = R \rightarrow 0$$ whose maps are homogeneous of degree $0$ and which is exact in positive degrees. Then the $i$-th Betti number $\beta_i(K[S_n])$ is the rank of $F_i$, which turns out to be independent of the choice of minimal resolution (see \cite[\S 1]{Eisenbud}). Note that $\beta_0(K[S_n])$ is always $1$ and $\beta_1(K[S_n])$ is the minimal number of generators for $I_{S_n}$, so these quantities are independent of the choice of $K$, but the higher Betti numbers do sometimes depend on the characteristic of $K$ (see \cite{BrunsHerzog}). 

The difficulty of directly applying the tool used in the previous section is that generally in first-order logic we have no way of expressing a property such as ``$x_1, \ldots, x_r$ constitute a minimal generating set for the ideal $I$'' unless it happens that we know an \emph{a priori} bound on the size $r$ of a minimal set of generators. However, even in the case of a parametric numerical semigroup generated by four quadratic polynomials, the size of the first Betti number may grow without bound as $n$ increases, as shown by a family of examples discovered by Bresinsky \cite{Bresinsky} which we will discuss in Section 5.  

To show that the Betti numbers $\beta_i(K[S_n])$ are EQP (fixing $i$, letting $n$ vary), we use a topological characterization by Bruns and Herzog. For any $\lambda \in S_n$, let the \emph{squarefree divisor complex} $\Delta_n(\lambda)$ be the simplicial complex with vertices $\{1, \ldots, k\}$ corresponding to the generators of $S_n$, and whose faces are the subsets $\{i_1, \ldots, i_r\}$ such that $\lambda - f_{i_1}(n) - \ldots - f_{i_r}(n) \in S_n$. We will compute the \emph{reduced} homologies of such complexes, so that $\dim_K \widetilde{H}_0(\Delta_n(\lambda); K)$ is the number of connected components of $\Delta_n(\lambda)$ minus $1$. Now by a theorem of Bruns and Herzog (\cite{BrunsHerzog}), the \emph{graded} $i$-th Betti number $\beta_{i,\lambda}(K[S_n])$ satisfies 

\begin{equation}
\label{BH}
\beta_{i,\lambda}(K[S_n]) = \dim_K \widetilde{H}_{i-1}(\Delta_n(\lambda); K),
\end{equation}

and the relation between these and the ungraded (``coarse'') Betti numbers is 
\begin{equation}
\label{coarse}
\beta_i(K[S_n]) = \sum_{\lambda \in S_n} \beta_{\lambda, i}(K[S_n]).
\end{equation}

\begin{proof}[Proof of Theorem~\ref{main_affine}]
Let $\Delta^1, \ldots, \Delta^N$ list all of the different simplicial complexes with vertex set $\{1, \ldots, k\}$ (so $N$ is no more than $2^{2^k}$). For any $j \in \{1, \ldots, N\}$, there is a parametric Presburger formula $\varphi^n_j(\textbf{x})$ such that for any $x \in \N$, $\varphi^n_j(\textbf{x})$ is true if and only if the squarefree divisor complex $\Delta_n(\textbf{x})$ is $\Delta^j$. That is, $\varphi^n_j(\textbf{x})$ is simply a conjunction over all possible subsets $\{i_1, \ldots, i_r\}$ of $[k]$ expressing that either $$\textbf{x} - f_{i_1}(n) - \ldots - f_{i_r}(n) \in S_n$$ or $$\textbf{x} - f_{i_1}(n) - \ldots - f_{i_r}(n) \notin S_n,$$ depending on whether $\{i_1, \ldots, i_r\}$ is a face in $\Delta^j$.

For any fixed $i, \ell \in \N$, let $$\Gamma_{i,\ell} = \{ j \in [N] \, : \,  \dim_K \widetilde{H}_{i-1}(\Delta^j; K) = \ell\}.$$ Now if $j \in \Gamma_{i, \ell}$ and $\ell > 0$, then for any $n$ the set $$\varphi^n_j(\N^m) := \{\textbf{x} \in \N^m \, : \, \varphi^n_j(\textbf{x}) \textup{ is true }\}$$ must be finite (since the $i$-th Betti number is finite!), and so for such a $j$, the quantity $| \varphi^n_j(\N^m) |$ is EQP as a function of $n$.

Hence by Equations~(\ref{BH}) and (\ref{coarse}) above, $$\beta_i(K[S_n]) = \sum_{\lambda \in S_n} \dim_K \widetilde{H}_{i-1}(\Delta_n(\lambda); K) = \sum_{\ell > 0} \sum_{j \in \Gamma_{i,\ell}} \ell \cdot | \varphi^n_j(\N^m) |, $$ which is a finite sum of EQP functions and hence is EQP.
\end{proof}

\section{Degrees of parametric semigroup invariants}
Theorems \ref{main_numerical} and \ref{main_affine} state that many invariants are eventually quasi-polynomial, but do not give any bounds on these EQP functions. In general it is difficult to extract degree bounds whenever we prove results via parametric Presburger arithmetic and Theorem~\ref{BGW}. However, some degree bounds are known or can be easily obtained from results in the literature.

For instance, in Remark 5.10 of \cite{KO}, Kerstetter and O'Neill establish the eventual periodicity of the first Betti number $\beta_1(S_n)$ of a numerical semigroup in the case where the generators $f_i(n)$ are linear functions of $n$. That is, they show for such parametric families that $\beta_1(S_n)$ is EQP \emph{of degree zero}. Our conclusion is weaker because we have no degree bound, but our context is much more general (affine semigroups instead of numerical semigroups, the generators of the semigroup can be given by arbitrary polynomials, and it applies to Betti numbers of any homological degree).

Similarly, our EQP result for the cardinality of the delta set (Theorem~\ref{main_numerical} part (2) (e)) is both more general, and weaker, than similar results obtained by Chapman, Kaplan, Lemburg, Niles, and Zlogar in \cite{CKLNZ}: they consider only a special case of shifted monoids (in which $f_i(n)$ has the form $n + a_i$), but they prove that in this case $|\Delta(S_n)|$ is eventually the constant value of $1$. 

Returning to Betti numbers of parametric numerical semigroups, we note that it is  possible to obtain an upper bound on the degree of the EQP functions for the Betti numbers by an elementary argument:


\begin{prop}
Let $S_n = \langle \fvec_1(n), \dots, \fvec_k(n) \rangle$ be a parametric numerical semigroup. Then for each $i$, the degree of the EQP function for $\beta_i(K[S_n])$ is at most $\sum_j \deg(f_j)$. 
\end{prop}
\begin{proof}
   For the Frobenius number $\textup{F}(f_1(n), \ldots, f_n(n))$ of a parametric numerical semigroup, we can give the obvious crude upper bound of $\prod_j f_j(n)$, which has degree $\sum_j \deg(f_j)$ Now using the Bruns-Herzog homological characterization of Betti numbers, one can reason that if $\lambda > \textup{F}(S_n) + i \cdot \max(\fvec_1(n), \ldots, \fvec_k(n))$, then after subtracting any $i$ of the semigroup generators $\fvec_1(n), \ldots, \fvec_k(n)$ from $\lambda$ we must still have an element of $S_n$, hence $\Delta_n(\lambda)$ is the full $k-1$-simplex and the graded Betti number $\beta_{i, \lambda}$ is $0$. Therefore the degree of the EQP function for $\beta_i(K[S_n])$ is no greater than $\sum_j \deg(f_j)$.
\end{proof}

In the special case of the first Betti number, the bound can be substantially improved. In particular, in this case the degree bound is independent of the number of generators $k$.  
\begin{thm}\cite[Theorem 8.26]{Rosales_GS}
  Let $S = \langle a_1, \dots, a_k \rangle$ be a numerical semigroup where $a_1 \leq \dots \leq a_k$. Then the cardinality of a minimal generating set of $I_S$ is at most
$$ \frac{(2a_1-k+1)(k-2)}{2} + 1.$$
\end{thm}
\begin{cor} \label{B1bound}
Let $S_n = \langle f_1(n), \dots, f_k(n) \rangle$ be a parametric numerical semigroup where $\deg(f_1) \leq \deg(f_2) \leq \dots \leq \deg(f_k)$. Then the degree of the EQP function for $\beta_1(K[S_n])$ is at most $\deg(f_1)$. 
\end{cor}

When all of the polynomials $f_1(n), \dots, f_k(n)$ are linear, this bound is not sharp; it would say that $\beta_1(K[S_n])$ is at most eventually quasilinear (that is, EQP of degree 1) but in fact the Kerstetter and O'Neill result shows that it is quasiperiodic (that is, quasipolynomial of degree 0.) 

For semigroups $S_n$ generated by polynomials $f_i(n)$ of degree greater than $1$, no sharp bounds on the degree of $\beta_1(S_n)$ appear to be known. We recall an example of Bresinsky \cite{Bresinsky}, which in our notation is the family of semigroups $$\langle 4n^2 - 2n, \, 4n^2 - 1, \, 4n^2 + 2n, \, 4n^2 + 4n - 1 \rangle$$ generated by four quadratic polynomials. Bresinsky showed that the ideal of relations $I_{S_n}$ cannot be generated by fewer than $2n$ elements, and therefore the degree of the EQP function $\beta_1(K[S_n])$ is at least $1$ (and at most $2$).\footnote{This would appear to be a counterexample to a conjecture of Kerstetter and O'Neill in \cite{KO} that the first Betti number of a parametric numerical semigroup is always eventually periodic.}

We now show how Bresinsky's example can be generalized: for any even $d$, there is a parametric numerical semigroup $S_n$ of degree $d$ such that $\beta_1(K[S_n])$ has degree at least $d/2$.

\begin{thm}
\label{Bres_general} Fix $d \geq 2$ even. Let $S_n = \langle a_1(n), a_2(n), a_3(n), a_4(n) \rangle$ where:

$$a_1(n) = 4 n^d - 2 n^{d/2},$$
$$a_2(n) = 4 n^d - 1,$$
$$a_3(n) = 4 n^d + 2 n^{d/2},$$
$$a_4(n) = 4 n^d + 4 n^{d/2} - 1.$$

Then $\beta_1(K[S_n]) \geq 2 n^{d/2}$.
\end{thm}

\begin{proof}
We will use the Bruns-Herzog method described in the previous section, according to which it suffices to produce $2 n^{d/2}$ distinct elements $f(\mu) \in S_n$ (where $1 \leq \mu \leq 2 n^{d/2}$) such that the squarefree divisor complex $\Delta_{f(\mu)}$ is disconnected. In fact, for every such $\mu$, the complex $\Delta_{f(\mu)}$ consists of two disjoint $1$-faces.

For integers $\mu \in [1, 2 n^{d/2}]$, we define \begin{equation}\label{eq:f_mu}f(\mu) = (\mu + 1)\cdot a_1 + (2 n^{d/2} - \mu) \cdot a_2.\end{equation}

Direct calculation shows that 

\begin{equation}\label{eq:f_mu_2} f(\mu) = (\mu - 1) \cdot a_3 + (2 n^{d/2} - \mu) \cdot a_4, \end{equation}

so the complex $\Delta_{f(\mu)}$ contains the two $1$-faces corresponding to $\{a_1, a_2\}$ and $\{a_3, a_4\}$. In fact we will show that $\Delta_{f(\mu)}$ contains no other $1$-faces.

Let $M = 2 n^{d/2} - 1$. Then we have $a_2 = a_1 + M$ and $a_4 = a_3 + M,$ from which it follows immediately that

\begin{equation}\label{eq:M2}f(\mu+1) = f(\mu) - M.\end{equation}

Noting that $a_3 = a_2 + M + 2$, we observe that if $r \in \{1,2\}$ and $s \in \{3,4\}$ then

\begin{equation}\label{eq:M3} f(\mu) - a_r - a_s \equiv -2 \,\, (\textup{mod } M). \end{equation}

\begin{claim}
\label{coeff_bound}
If $x = \sum_i z_i \cdot a_i$ where $z_1, z_2, z_3, z_4 \in \N$ and $x \equiv -2 \, \, (\textup{mod } M),$ then $z_3 + z_4 \geq M - 1$.
\end{claim}

\begin{proof}
Since $a_1 = 4 n^d - 2 n^{d/2} = (2 n^{d/2} - 1) \cdot 2 n^{d/2}$, it is a multiple of $M$, and the respective residues of $a_1, a_2, a_3,$ and $a_4$ modulo $M$ are $0, 0, 2,$ and $2$. Since $M$ is odd, the minimal number of times we can add $2$ to itself and reach a number congruent to $-2$ modulo $M$ is $M-1$, and the Claim follows.
\end{proof}

Now our goal is to show that for any integer $\mu \in [1, 2n^{d/2}]$, $r \in \{1,2\}$, and $s \in \{3,4\}$, we have

\begin{equation}\label{eq:sr}f(\mu) - a_r - a_s \notin S_n.\end{equation} 

Once we have this, it follows from our previous discussion that $\Delta_{f(\mu)}$ is disconnected and hence the Theorem will be established.

So fix such a $\mu \in [1, 2 n^{d/2}]$, $r \in \{1,2\}$, and $s \in \{3,4\}$. By (\ref{eq:M2}), $f(\mu) \leq f(1)$, so since $a_1 < a_2$ and $a_3 < a_4$, we have
\begin{align*}
f(\mu) - a_r - a_s &\leq f(1) - a_1 - a_3 \\
 &= 0 \cdot a_3 + (2 n^{d/2} - 1) \cdot a_4 - a_1 - a_3\\
 &= M \cdot a_4 - a_1 - a_3\\
 & = M \cdot (a_3 + M) - a_1 - a_3\\
 & = (M-1) \cdot a_3 + M^2 - a_1\\
 & < (M-1) \cdot a_3.
\end{align*}
where the last inequality follows because $a_1 = M(M+1) > M^2$. 

Suppose, towards a contradiction, that $f(\mu) - a_r - a_s \in S_n$. Noting that the numbers $f(\mu),$ $a_1,$ and $a_2$ are all divisible by $M$ while $a_3$ and $a_4$ are congruent to $2$ modulo $M$, we have $f(\mu) - a_r - a_s \equiv -2 \, \, (\textup{mod } M)$, so applying Claim~\ref{coeff_bound},

\begin{equation}\label{eq:lowerbound} f(\mu) - a_r - a_s \geq (M-1) \cdot a_3. \end{equation} But this contradicts the upper bound on $f(\mu) - a_r - a_s$ which we just established.
\end{proof}

To summarize, let $g(d)$ be the maximum degree of any EQP function that can arise as $\beta_1(K[S_n])$ where $S_n$ is a parametric numerical semigroup whose generators are all polynomials of degree $d$. Combining Corollary~\ref{B1bound} with Theorem~\ref{Bres_general}, we obtain that if $d$ is even, then
$$ d/2 \leq g(d) \leq d.$$

\begin{quest}
What is the precise value of $g(d)$? 
\end{quest}


\begin{quest}
  For $i > 1$, is there an upper bound on the degree of the $i$-th Betti number of a parametric numerical semigroup which is independent of the number $k$ of generators?
\end{quest}

\end{document}